\documentclass[11pt,oneside]{amsart}

\usepackage{amsmath,amsthm,amssymb,amsfonts,amscd}
\usepackage{url}
\usepackage{color}
\usepackage{setspace}
\usepackage{enumerate}

\usepackage{epsfig}
\usepackage{graphicx}
\usepackage{wrapfig}
\usepackage{mathrsfs}
\usepackage{pinlabel}

\definecolor{grey}{rgb}{0.7,0.7,0.7}

\newcounter{notes}%

\newcommand{\ignore}[1]{}

\theoremstyle{plain}
\newtheorem{theorem}{Theorem}
\newtheorem{proposition}[theorem]{Proposition}

\newtheorem{fact}[theorem]{Fact}
\newtheorem{lemma}[theorem]{Lemma}

\newtheoremstyle{theoremwithref}{}{}{\itshape}{}{\bfseries}{.}{.5em}{#1 #2 #3}
\theoremstyle{theoremwithref}

\theoremstyle{definition}

\newtheorem{conjecture}[theorem]{Conjecture}
\newtheorem{criterion}[theorem]{Criterion}

\numberwithin{theorem}{section}
\numberwithin{equation}{section}

\title{Strip maps of small surfaces are convex}

\author{Fran\c{c}ois Gu\'eritaud}
\address{CNRS and Universit\'e Lille 1, Laboratoire Paul Painlev\'e, 59655 Villeneuve d'Ascq Cedex, France
\newline Wolfgang-Pauli Institute, University of Vienna, CNRS-UMI 2842, Austria}
\email{francois.gueritaud@math.univ-lille1.fr}

\thanks{
Partially supported by the Agence Nationale de la Recherche under the grants DiscGroup (ANR-11-BS01-013) and ETTT (ANR-09-BLAN-0116-01), and through the Labex CEMPI (ANR-11-LABX-0007-01). 
}

\begin{document}

\begin{abstract}
The strip map is a natural map from the arc complex of a bordered hyperbolic surface $S$ to the vector space of infinitesimal deformations of $S$. We prove that the image of the strip map is a convex hypersurface when $S$ is a surface of small complexity: the punctured torus or thrice punctured sphere. 
\end{abstract}

\maketitle

\section{Introduction}
Let $S$ be a compact orientable surface of genus $g\geq 0$ with $p\geq 1$ boundary components, where $2g+p\geq 3$. The arc complex of $S$ is the complex $\overline{X}$ whose vertices are the isotopy classes of non-boundary-parallel embedded arcs in $S$ with endpoints in $\partial S$, and whose $(k-1)$-cells (for $2\leq k \leq 6g-6+3p=:N$) correspond to $k$-tuples of mutually nonisotopic arcs that can be embedded in $S$ disjointly. In this paper we study some realizations of $\overline{X}$ in $\mathbb{R}^N$ arising from hyperbolic geometry.

The top-dimensional cells of $\overline{X}$ correspond to so-called hyperideal triangulations of $S$, namely, collections of arcs subdividing $S$ into disks each of which is bounded by three segments of $\partial S$ and three arcs. Elements of $\overline{X}$ can always be represented in barycentric coordinates in the form $\sum_{i=1}^N \lambda_i\alpha_i$ where the $\lambda_i$ are nonnegative reals summing to $1$ and the $\alpha_i$ are arcs of a hyperideal triangulation. Note that $\overline{X}$ is infinite unless $S$ is the thrice punctured sphere.

A cell of $\overline{X}$ (of any dimension) is called \emph{small} if the arcs corresponding to its vertices \emph{fail} to decompose $S$ into disks. For example, vertices of $\overline{X}$ are small cells but top-dimensional cells are not. An important result of Harer and (independently) Penner \cite{harer, penner} is the following: \emph{the complement $X\subset\overline{X}$ of the union of all small cells is homeomorphic to an open $(N-1)$-ball}. Up to boundary effects, we may therefore think of the infinite complex $\overline{X}$ as (essentially) a ball.

It is an interesting question whether this triangulation of the ball can be realized by affine simplices in $\mathbb{R}^{N-1}$ as a tiling of, say, a convex region. One of the main results of \cite{dgk2} is an affirmative answer: 
\begin{proposition} \label{prop:stripconvex}
The projectivized \emph{strip map} (defined below) associated to a hyperbolic metric on $S$ restricts to an embedding of $X$ into $\mathbb{P}(\mathbb{R}^N)$, whose image is a convex open set with compact closure in some affine chart. 
\end{proposition}

\subsection{The strip map} \label{sec:stripmap}
Let $\mathcal{T}$ be the space of hyperbolic metrics on $S$ with totally geodesic boundary, seen up to isotopy. Then $\mathcal{T}$, also called the Teichm\"uller space, is diffeomorphic to an open $N$-ball. Let $g\in\mathcal{T}$ be a fixed metric and $x=\sum_{i=1}^N \lambda_i\alpha_i$ a point of $\overline{X}$. We consider for each arc $\alpha\in\overline{X}^{(0)}$ its geodesic representative in $(S,g)$, still denoted $\alpha$, that exits $\partial S$ perpendicularly: in particular, the (representatives of the) $\alpha_i$ are disjoint. Suppose moreover that for each $\alpha\in \overline{X}^{(0)}$ we are given a point $p_\alpha\in\alpha$, called the \emph{waist}.
To any reals $c_1\dots,c_N\geq 0$ we can then associate a deformation $\mathsf{Strip}\left ( g, \sum_{i=1}^N c_i \alpha_i \right )\in\mathcal{T}$, as follows:
\begin{itemize}
 \item Glue funnels to $\partial S$, turning $(S,g)$ into an infinite-area hyperbolic surface $S'$ without boundary;
 \item For each $1\leq i \leq N$, cut $S'$ open along the geodesic $\alpha_i'$ that extends~$\alpha_i$;
 \item Insert along $\alpha'_i$ a \emph{strip} of $\mathbb{H}^2$ of width $c_i$, i.e.\ the region bounded by two geodesics of $\mathbb{H}^2$ perpendicular to a segment of length $c_i$ at its endpoints. Make sure these endpoints become glued to the two copies of the waist $p_{\alpha_i}\in\alpha_i'$ obtained after cutting $\alpha_i'$  open.
 \item Define $\mathsf{Strip}\left ( g, \sum_{i=1}^N c_i \alpha_i \right )$ as the convex core of the new surface with $N$ strips inserted.
\end{itemize}

We may now define a continuous map associated to $g\in\mathcal{T}$ and to the chosen system of waists $(p_\alpha)_{\alpha\in\overline{X}^{(0)}}$:
$$ \begin{array}{rrcl} \boldsymbol{f}: & \overline{X} & \longrightarrow & T_{[g]}\mathcal{T} \\ 
& \sum_1^N \lambda_i \alpha_i & \longmapsto & \left . \frac{\mathrm{d}}{\mathrm{d}t} \right |_{t=0} 
\mathsf{Strip}\left ( g, \sum_{i=1}^N t\lambda_i \alpha_i \right ). \end{array} $$
This map $\boldsymbol{f}$, called the (infinitesimal) strip map, is the main object of interest in this paper. Its projectivization $f:\overline{X} \rightarrow \mathbb{P}(T_{[g]}\mathcal{T})\simeq \mathbb{P}(\mathbb{R}^N)$ is the projectivized strip map mentioned in Proposition \ref{prop:stripconvex}. The strip construction goes back at least to Thurston \cite{thu86a}; see also \cite{pt10}. 

Remarkably, the set $f(X)$ is actually independent of the choices of waists. In fact $f(X)$ coincides with the projectivization of the space of infinitesimal deformations of the hyperbolic metric on $S$ such that all closed geodesics become (in a strict sense) shorter to first order \cite{dgk2}. This has important consequences concerning the structure of the deformation space of Margulis spacetimes (quotients of $\mathbb{R}^{2,1}$ by free groups acting properly discontinuously), and motivates a more detailed study of $\boldsymbol{f}$.

\subsection{Convex hypersurfaces}
Proposition \ref{prop:stripconvex} can be rephrased thus: for any two top-dimensional simplices of $\overline{X}$ with vertex lists $(\alpha,\beta_1,\dots, \beta_{N-1})$ and $(\alpha',\beta_1,\dots, \beta_{N-1})$, there exist reals $A,A',B_1,\dots, B_N$ such that
\begin{itemize}
 \item $(\boldsymbol{f}(\alpha),\boldsymbol{f}(\beta_1),\dots, \boldsymbol{f}(\beta_{N-1}))$ 
 is a basis of $\mathbb{R}^N$;
 \item $\displaystyle{A\boldsymbol{f}(\alpha)+A'\boldsymbol{f}(\alpha')=\sum_{i=1}^{N-1} B_i\boldsymbol{f}(\beta_{i})}$;
 \item $\sum_{i=1}^{N-1} B_i >0$ and $A,A'>0$.
\end{itemize}
(The first two conditions already imply that $(A,A',B_1,\dots, B_{N-1})$ are unique up to scaling.)
The following conjecture appears in \cite{dgk2}:
\begin{conjecture} \label{conj:stripconvex}
For an appropriate choice of waists $(p_\alpha)_{\alpha\in\overline{X}^{(0)}}$, the image of $\boldsymbol{f}|_X$ in $T_{[g]}\mathcal{T}$ is a convex hypersurface, with codimension-1 edges looking salient from the origin. In other words (see Figure \ref{fig:convexity}), the numbers $A,A', B_i$ defined above satisfy the extra condition
$\displaystyle{A+A'<\sum_1^N B_i}$.
\end{conjecture}
\begin{figure}[ht!]
\labellist
\small\hair 2pt
\pinlabel {$\boldsymbol{f}(\alpha)$} at -4 74
\pinlabel {$\boldsymbol{f}(\alpha')$} at 163 73
\pinlabel {$\boldsymbol{f}(\beta_1)$} at 53 31.5
\pinlabel {$\boldsymbol{f}(\beta_2)$} at 84 74
\pinlabel {$0$} at 77 1
\endlabellist
\centering
\includegraphics[width=7cm]{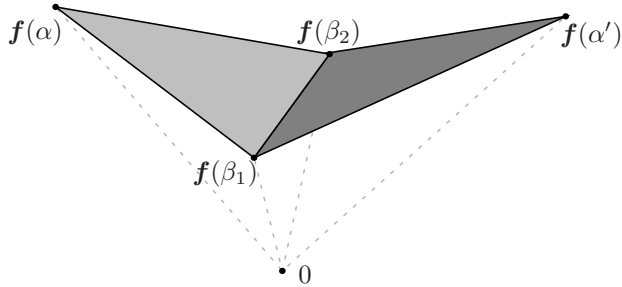}
\caption{A convex hypersurface in $\mathbb{R}^3$.}
\label{fig:convexity} 
\end{figure}
Since $X$ is dense in $\overline{X}$, restriction to $X$ is inessential in Conjecture \ref{conj:stripconvex}; it is only meant to ensure the image is a (noncomplete) topological submanifold.
Conjecture \ref{conj:stripconvex} would give a realization of $X$ within the simplicial decomposition arising from the convex hull of a discrete set $\boldsymbol{f}(\overline{X}^{(0)})$. It is not clear a priori that such convex realizations should exist, even given Proposition \ref{prop:stripconvex}.

Note that Conjecture \ref{conj:stripconvex} has a well-studied finite counterpart: the complex of diagonal subdivisions of a (finite, planar, convex) $n$-gon is finite, and is realized as the cell decomposition of the (dual) \emph{associahedron}, a now classical polytope in $\mathbb{R}^{n-3}$: see for example \cite{loday} and the references therein. In this note, we prove

\begin{theorem} \label{thm:main}
Conjecture \ref{conj:stripconvex} is true for $S$ a once punctured torus or a thrice punctured sphere.
\end{theorem}

The proof will be a rather explicit computation. The once punctured torus and the thrice punctured sphere are called the \emph{small} (orientable) surfaces; their arc complexes are planar triangle complexes recalled in Section \ref{sec:trees}. As these complexes are dual to trees, it is not hard to realize them in the boundaries of convex (finite or infinite) polyhedra of $\mathbb{R}^3$, so Theorem \ref{thm:main} is not a new \emph{realizability} result. However, 
\begin{itemize}
 \item It is interesting to note that the strip map gives a natural realization.
 \item In the case of the punctured torus, we can extend Theorem \ref{thm:main} to singular hyperbolic metrics (Theorem \ref{thm:ellip}), replacing the boundary component with a cone point of angle $\theta\in (0,2\pi)$. Proposition \ref{prop:stripconvex} was already extended to that singular context in \cite{PT-note}. Theorem \ref{thm:parab} also treats the intermediate case of a cusped metric ($\theta=0$).
 \item In the case of the thrice punctured sphere, we will see that a naive choice of waists, such as the midpoints of the arcs, does in general \emph{not} work for Conjecture \ref{conj:stripconvex}. This could shed light on the general~case.
\end{itemize}

\subsection{Plan} Section \ref{sec:background} contains reminders about the geometry of strip deformations, the arc complexes of the small surfaces, and hyperbolic geometry (Killing fields and the Minkowski model). Section \ref{sec:S03} proves Theorem \ref{thm:main} for the thrice punctured sphere, and Section \ref{sec:S11} for the once punctured torus.

\section{Background} \label{sec:background}
\subsection{The sine formula}
To estimate the effect of a strip deformation on the metric of $S$, it is convenient to compute how it affects the lengths of various geodesics. Here we give a formula: the proof is similar to the classical \emph{cosine formula} for earthquake deformations \cite{ker83}, and can be found in \cite[\S2.1]{dgk2}.

For simplicity, we restrict to strip deformations $\boldsymbol{f}(\alpha)$ along a single arc~$\alpha$: the general case $\boldsymbol{f}(\sum_1^N \lambda_i \alpha_i)$ is then recovered by linearity. Let $\gamma \subset S$ be a closed geodesic, and $\mathrm{d}\ell_\gamma:T_{[g]}\mathcal{T}\rightarrow \mathbb{R}$ the differential of its length function. Suppose that $\gamma$ intersects $\alpha$ at points $q_1,\dots, q_n$ lying at distances $r_1,\dots, r_n\geq 0$ from the waist $p_\alpha$, measured along the arc $\alpha$. Then
\begin{equation}
 \label{cosineformula}
 \mathrm{d}\ell_\gamma(\boldsymbol{f}(\alpha))=\sum_{i=1}^{n} \sin ( \measuredangle_{q_i}(\alpha, \gamma)) \, \cosh (r_i) 
\end{equation}
where $\measuredangle_{q_i}(\alpha, \gamma) \in (0,\pi)$ denotes the angle, at the point $q_i$, between the directions of $\alpha$ and $\gamma$.

This formula shows for example that a strip deformation along a very long arc $\alpha$ will have a huge lengthening effect on the boundary length of $S$ --- more precisely, on the lengths of the boundary components of $S$ that $\alpha$ intersects but that lie far away from the waist $p_\alpha$.

\subsection{Arc complexes of small surfaces}\label{sec:trees}
In a (hyperideal) triangulation~$\tau$ of the surface $S$, whenever an arc $\alpha$ separates two distinct regions, removing $\alpha$ creates a hyperideal quadrilateral of which $\alpha$ was a diagonal. The triangulation obtained by inserting back the other diagonal is called the \emph{diagonal flip} of $\tau$ at $\alpha$. Two distinct top-dimensional faces of the arc complex $\overline{X}$ share a codimension-1 face exactly when the two corresponding triangulations of $S$ are related by a diagonal flip.
\subsubsection{The thrice punctured sphere}\label{sec:arcs03}
The thrice punctured sphere $S$ has one triangulation $\tau$ obtained by connecting all pairs of distinct punctures together. It also has three more triangulations, obtained from $\tau$ by flipping one of its $3$ edges. In total, the arc complex $\overline{X}$ has $6$ vertices, $9$ one-cells ($3$ of them inner), and $4$ two-cells (the triangulations). The full mapping class group of $S$ has order 12 and projects to the automorphism group of~$\overline{X}$, which is the order-6 dihedral group. The kernel is the reflection of $S$ preserving the arcs of $\tau$ pointwise. The dual of $\overline{X}$ is a $3$-branched star. See Figure \ref{fig:smallsurf}.
\subsubsection{The once punctured torus}
Up to the action of the mapping class group $\mathrm{GL}_2(\mathbb{Z})$, the punctured torus $S$ of interior $\simeq (\mathbb{R}^2\smallsetminus \mathbb{Z}^2)/\mathbb{Z}^2$ has only one hyperideal triangulation, obtained e.g.\ by projecting to $S$ the three segments of $\mathbb{R}^2\smallsetminus \mathbb{Z}^2$ connecting the origin to $(1,0)$, $(0,1)$, and $(1,1)$. The resulting arc complex $\overline{X}$ is dual to an infinite planar trivalent tree, with one vertex for each rational number $p/q\in\mathbb{P}^1(\mathbb{Q})$ (corresponding to the segment from the origin to $(p,q)$). The mapping class group maps onto the automorphism group of $\overline{X}$, with kernel $\{\mathrm{Id},-\mathrm{Id}\}$. See Figure \ref{fig:smallsurf}.
\begin{figure}[ht!]
\labellist
\small\hair 2pt
\pinlabel {$\tau$} at 136 96
\endlabellist
\centering
\includegraphics[width=12cm]{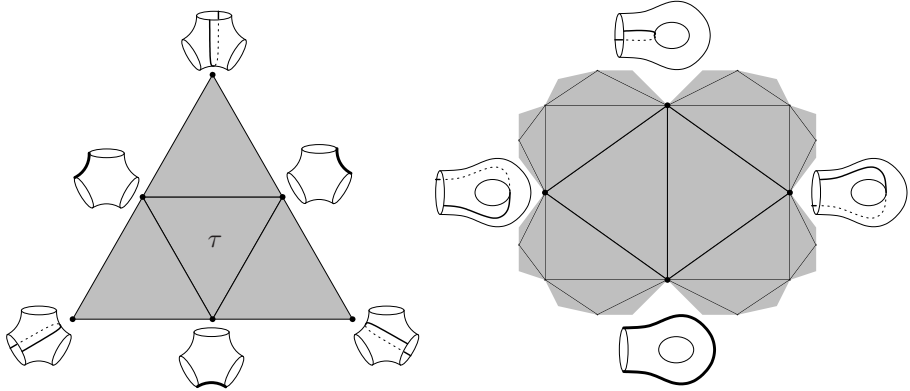}
\caption{The arc complexes $\overline{X}$ of the two small surfaces.}
\label{fig:smallsurf} 
\end{figure}

\subsection{Lorentzian geometry}
We see $G:=\mathrm{PSL}_2(\mathbb{R})$ as the isometry group of the hyperbolic plane $\mathbb{H}^2$, and the Lie algebra $\mathfrak{g}:=\mathfrak{psl}_2(\mathbb{R})$ as the space of Killing vector fields on $\mathbb{H}^2$. The Killing form on $\mathfrak{g}$, multiplied by $\frac{1}{2}$, makes $\mathfrak{g}$ isometric to Minkowski space $(\mathbb{R}^{2,1},\langle \cdot |\cdot \rangle)$. Viewing $\mathbb{H}^2$ as one sheet (call it ``future'') of the unit hyperboloid of $\mathfrak{g}$, we can then identify the isometry action of $G$ on $\mathbb{H}^2$ with the adjoint action. For $\mathcal{Y}\in \mathfrak{g}$, we write $\Vert \mathcal{Y} \Vert :=\sqrt{\langle \mathcal{Y} | \mathcal{Y} \rangle}$ and let $d_{\mathbb{H}^2}$ be the hyperbolic distance function.
\begin{fact} \label{fact:classical}
The following are classical:
\begin{enumerate}
 \item If $\mathcal{Y},\mathcal{Z}\in\mathbb{H}^2\subset\mathfrak{g}$ then 
 $\Vert \mathcal{Y}-\mathcal{Z} \Vert = 2\sinh (d_{\mathbb{H}^2}(\mathcal{Y},\mathcal{Z})/2)$.
 \item If $\mathcal{Y},\mathcal{Z}\in\mathfrak{g}$ satisfy $\Vert \mathcal{Y} \Vert^2 = \Vert \mathcal{Z} \Vert^2 = 1$ and the hyperbolic half-planes $P_\mathcal{Y}:=\{u\in \mathbb{H}^2\, |~ \langle u|\mathcal{Y} \rangle \geq 0\}$ and $P_\mathcal{Z}:=\{u\in \mathbb{H}^2\, |~ \langle u|\mathcal{Z} \rangle \geq 0\}$ are disjoint, then $\Vert \mathcal{Y}-\mathcal{Z} \Vert = 2\cosh (d_{\mathbb{H}^2}(P_\mathcal{Y},P_\mathcal{Z})/2)$.
 \item If $\mathcal{Y},\mathcal{Z}\in\mathfrak{g}$ are future-pointing lightlike (i.e.\ isotropic) vectors representing ideal points $y,z\in \partial_\infty \mathbb{H}^2$, a Killing field $\mathcal{U}\in\mathfrak{g}$ belongs to $\mathbb{R}^{>0} \mathcal{Y} - \mathbb{R}^{>0} \mathcal{Z}$ if and only if $\mathcal{U}$ represents an infinitesimal translation of axis perpendicular to the hyperbolic line $yz$, with $y$ to the left and $z$ to the right of the axis. The velocity of that Killing field along its axis is then just $\Vert \mathcal{U} \Vert$.
\end{enumerate}
\end{fact}

\subsection{Convexity criterion}
We can use Killing fields to express the local convexity of the hypersurface $\boldsymbol{f}(X)$ at a codimension-1 face, as follows.
\subsubsection{The thrice punctured sphere} \label{sec:crit03}
For $(S,g)$ a hyperbolic thrice punctured sphere, let $\alpha, \beta, \gamma$ be the arcs of the triangulation $\tau$ of Section \ref{sec:arcs03} and let $\delta$ be the arc obtained by flipping $\alpha$ in $\tau$. 

Note that $(\alpha, \beta, \gamma)$ and $(\beta, \gamma, \delta)$ are top-dimensional faces of the arc complex $\overline{X}$. Let us consider local convexity at the edge $\boldsymbol{f}([\beta, \gamma])=\boldsymbol{f}([\alpha, \beta, \gamma])\cap\boldsymbol{f}([\beta, \gamma, \delta])$, corresponding to the flip that replaces $\alpha$ with $\delta$. By the discussion\footnote{Proposition \ref{prop:stripconvex}, which informs this discussion, is also easily verifiable by hand here.} preceding Conjecture \ref{conj:stripconvex}, there exists a relationship of the form
\begin{equation}
B \boldsymbol{f}(\beta)+C \boldsymbol{f}(\gamma) - A \boldsymbol{f}(\alpha) - D \boldsymbol{f}(\delta)=0 \in T_{[g]}\mathcal{T}
\label{eq:vanish} 
\end{equation}
for some $(A,B,C,D)\in \mathbb{R}^4 \smallsetminus \{0\}$, unique up to scalar multiplication, and we can assume $B + C> 0$ and $A,D> 0$. Convexity at $\boldsymbol{f}([\beta, \gamma])$ is the property 
 \begin{equation}
A+D<B+C. \label{eq:tocheck}  
 \end{equation}

Lift all arcs $\alpha, \beta, \gamma, \delta$ to $\mathbb{H}^2$, obtaining a tiling $\mathcal{E}$ of $\mathbb{H}^2$ into infinitely many triangles (or ``tiles''), each with one right angle and two hyperideal vertices. This tiling is equivariant with respect to a holonomy representation 
$$\rho:\pi_1(S)\rightarrow \mathrm{PSL}_2(\mathbb{R})\simeq \mathrm{Isom}^+(\mathbb{H}^2).$$
The relationship \eqref{eq:vanish} expresses the fact that appropriate infinitesimal strip deformations on $\beta, \gamma$ can cancel out appropriate infinitesimal strip deformations on $\alpha, \delta$, yielding the \emph{trivial} deformation of $S$. This can be interpreted (see \cite[\S4]{dgk2}) as an assignment of a Killing field to each tile, via a map
$$\psi:\mathcal{E}\rightarrow \mathfrak{psl}_2(\mathbb{R})\simeq\mathrm{Kill}(\mathbb{H}^2)$$
satisfying the following properties:
\begin{enumerate}[(i)]
 \item Equivariance: for any tile $t\in\mathcal{E}$ and any $\eta\in \pi_1(S)$, we have $\psi(\eta\cdot t)= \mathrm{Ad}(\rho(\eta))(\psi(t))$; in other words $\psi$ defines a tilewise Killing field on the quotient $S$ of $\mathbb{H}^2$;
 \item Vertex consistency: if $t_1$, $t_2$, $t_3$, $t_4$ are the tiles adjacent to a lift of the vertex $\alpha \cap \delta$, numbered clockwise, then $\psi(t_1)-\psi(t_2)+\psi(t_3)-\psi(t_4)=0$; in other words, the $\psi(t_i)$ form a parallelogram in $\mathfrak{psl}_2(\mathbb{R})$;
 \item Edge increments: suppose the geodesic line  $\lambda$ of $\mathbb{H}^2$ is a lift of the arc $\beta$ (resp.\ $\gamma, \alpha, \delta$), and $p\in \lambda$ is the lift of the corresponding waist. If $\lambda$ separates two adjacent tiles $e,e'\in\mathcal{E}$, then $\psi(e')-\psi(e)$ is a Killing field representing an infinitesimal translation whose axis is the perpendicular to $\lambda$ through the lifted waist $p$, and whose signed velocity (measured towards $e'$) is the real number $B$ (resp.\ $C, -A, -D$).  
\end{enumerate}
The increment condition~(iii) expresses the fact that the relative motion of adjacent tiles is given by some strip deformation. The vertex condition~(ii) can be rephrased thus: the point $\alpha\cap \delta$ cuts $\alpha$ in two halves, but the increment of $\psi$ across either half is the same. Condition~(i) expresses the fact that the linear combination of all 4 (signed) strip deformations is trivial in $T_{[g]}\mathcal{T}$.

We can turn this Killing-field interpretation around:
\begin{criterion} \label{crit:killing}
Conversely, if we exhibit an assignment $\psi$ of Killing fields to tiles, satisfying (i)--(ii)--(iii) for some reals $A,B,C,D$ with $A,D>0$, then local convexity of  $\boldsymbol{f}(\overline{X})$ at the edge  $\boldsymbol{f}([\beta, \gamma])$ (where $\boldsymbol{f}$ is defined for the waists induced by the translation axes of the increments of~$\psi$) amounts to the inequality \eqref{eq:tocheck} above: $A+D<B+C$. 
\end{criterion}
In the rest of the paper, we will therefore check convexity of $\boldsymbol{f}$ by exhibiting special Killing fields and computing their velocities $A,B,C,D$.

\subsubsection{The once punctured torus}
The discussion of Section \ref{sec:crit03} is essentially unchanged when $S$ is a hyperbolic once-punctured torus and $\alpha, \beta, \gamma$ the arcs of a triangulation. The only difference is that the tiles are no longer right-angled in general, because $\alpha$ need not intersect its flip $\delta$ perpendicularly (unless $\beta, \gamma$ have equal lengths). This inconvenience is compensated by the fact that $\alpha, \delta$ intersect at their midpoints, which becomes a natural choice of waist.

\section{Proof of Theorem \ref{thm:main} for the thrice punctured sphere} \label{sec:S03}

In this section $S$ is the thrice punctured sphere.

\subsection{A bad choice of waists: midpoints}
We begin by remarking that, for some hyperbolic metrics $g$ on $S$, picking waists at the midpoints of the arcs would \emph{not} define a strip map $\boldsymbol{f}:\overline{X}\rightarrow T_{[g]}\mathcal{T}$ with convex image. Indeed, suppose $(S,g)$ has boundary components $a,b,c$ of lengths $0<\ell(a)\ll 1 =\ell(b)=\ell(c)$. Let $\alpha, \beta, \gamma, \delta$ denote the the arcs $bc,ca,ab,aa$ respectively, where an arc is referred to by the two boundary components it connects. Then $\ell(\beta)=\ell(\gamma)\gg 1 $ and $\ell(\alpha)$ is on the order of $1$: see Figure \ref{fig:wrong}.
\begin{figure}[ht!]
\labellist
\small\hair 2pt
\pinlabel {$\alpha$} at 201 39
\pinlabel {$\beta$} at 115 31
\pinlabel {$\gamma$} at 115 64
\pinlabel {$\delta$} at 130 54
\pinlabel {$p_\alpha$} at 202 48
\pinlabel {$p_\beta$} at 88 30
\pinlabel {$p_\gamma$} at 88 64
\pinlabel {$a$} at -2 50
\pinlabel {$b$} at 180 80
\pinlabel {$c$} at 180 14
\endlabellist
\centering
\includegraphics[width=10cm]{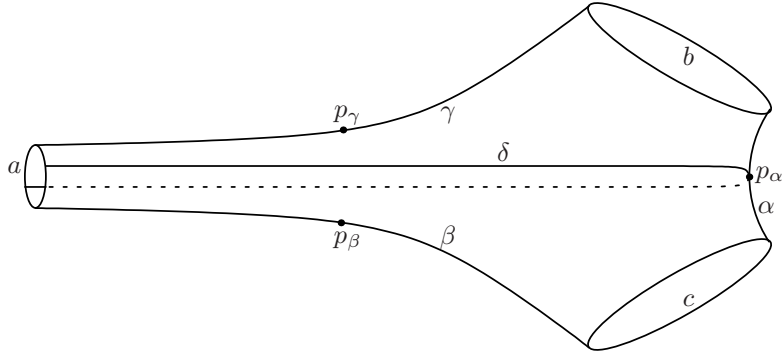}
\caption{A thrice punctured sphere with a short loop.}
\label{fig:wrong} 
\end{figure}

We know that there exist reals $A,B,C,D$ with $A,D>0$ satisfying \eqref{eq:vanish}. By symmetry, we can assume $B=C=1$. Let us prove that $A+D>2=B+C$, in violation of local convexity \eqref{eq:tocheck}.

The Teichm\"uller space $\mathcal{T}$ is coordinatized by the three boundary lengths $\ell(a),\ell(b),\ell(c)$, hence the range $T_{[g]}\mathcal{T}$ of $\boldsymbol{f}$ admits a dual basis $(\mathrm{d}\ell(a), \mathrm{d}\ell(b), \linebreak \mathrm{d}\ell(c))$. By \eqref{cosineformula}, the lengths of $b$ and $c$ are not affected by the infinitesimal deformation $\boldsymbol{f}(\delta)$, because $b\cap \delta = c\cap \delta=\emptyset$. They are affected at roughly unit rate by $\boldsymbol{f}(\alpha)$ because the arc $\alpha$ has length on the order of $1$ and intersects $b,c$. But they are affected at a \emph{huge} rate by $\boldsymbol{f}(\beta)$ and $\boldsymbol{f}(\gamma)$ because the waists on $\beta$ and $\gamma$ are far away from $b$ and $c$. So the identity $\mathrm{d}\ell(b)(\boldsymbol{f}(\beta)+\boldsymbol{f}(\gamma))=\mathrm{d}\ell(b)(A \boldsymbol{f}(\alpha)+D\boldsymbol{f}(\delta))$, true by \eqref{eq:vanish}, can only hold if $A$ is itself huge. Thus $A+D>2$, proving that $\boldsymbol{f}$ has nonconvex image.

\subsection{A good choice of waists}

In a general hyperbolic thrice-punctured sphere $S$, the arcs $\alpha, \delta$ intersect orthogonally (at the midpoint of $\delta$ but not of $\alpha$): we pick this point for the waists $p_\alpha$ and $p_{\delta}$, and do the same for the pair formed by $\beta$ (resp.\ $\gamma$) and its flip. Let us prove that under this choice, $\boldsymbol{f}$ has convex image.

The following is a hyperbolic generalization of a classical Euclidean fact.
\begin{lemma} \label{lem:bisector}
Let $\alpha_0, \alpha_1, \alpha_2$ be lines in $\mathbb{H}^2$ bounding half-planes with disjoint closures in $\mathbb{H}^2\cup\partial_\infty\mathbb{H}^2$ (i.e.\ the sides of a hyperideal triangle). Let $\beta_i$ be the common perpendicular of $\alpha_{i+1}$ and $\alpha_{i-1}$ (indices modulo $3$). The \emph{height} $h_i$ is the common perpendicular to $\beta_i$ and $\alpha_i$, intersecting $\alpha_i$ at the \emph{foot} $p_i$. Then the three heights $h_i$ are the inner angle bisectors of the triangle $p_0 p_1 p_2$.
\end{lemma}
\begin{proof}
By a compactness argument, there exist points $p'_i\in\alpha_i$ such that the triangle $p'_0 p'_1 p'_2$ has minimum possible perimeter. By Snell's law, $\alpha_i$ is the outer angle bisector at the vertex $p'_i$: so it is enough to prove that $p'_i=p_i$.

In Minkowski space $(\mathbb{R}^{2,1},\langle \cdot | \cdot \rangle)$, embed $\mathbb{H}^2$ as the upper unit hyperboloid. Let $v_i\in\mathbb{R}^{2,1}$ be the unit spacelike vector ($\langle v_i | v_i \rangle =1$) such that $\langle p'_{i+1}|v_i \rangle=0= \langle p'_{i-1}|v_i \rangle$ and $\langle p'_i|v_i \rangle>0$. By symmetry,  $\alpha_i=\mathrm{ker} \langle \, \cdot\,  | v_{i+1}+v_{i-1}\rangle \cap \mathbb{H}^2$, and $\mathrm{ker} \langle \, \cdot \, | v_{i+1}-v_{i-1}\rangle \cap \mathbb{H}^2$ is the line $h'_i$ perpendicular to $\alpha_i$ at $p'_i$.

Let $(w_0, w_1, w_2)$ be the dual basis to $(v_0, v_1, v_2)$, i.e. $\langle w_i|v_j\rangle = \delta_{ij}$. Then $w_{i+1}+w_{i-1}-w_i$ pairs to $0$ against $v_i+v_{i+1}$ and $v_i+v_{i-1}$ and $v_{i+1}-v_{i-1}$. This means that $\alpha_{i-1}, \alpha_{i+1}$ and $h'_i$ have a common perpendicular (necessarily $\beta_i$). Therefore $h'_i=h_i$, hence $p'_i=p_i$ as desired. (We may also note that all three heights $h_i=h'_i$ run through the point of $\mathbb{H}^2$ collinear with $w_0+w_1+w_2$, since that vector pairs to $0$ against $v_{i+1}-v_{i-1}$.)
\end{proof}

We now return to the thrice punctured sphere $S$. Let $\alpha, \beta, \gamma$ be the arcs connecting distinct boundary components; the waists $p_\alpha, p_\beta, p_\gamma$ are the feet of the heights of the hyperideal triangle with sides $\alpha, \beta, \gamma$. The point $p_\alpha=p_{\delta}$ is also the midpoint of the flipped edge $\delta$. Denote by $2\widehat{a}, 2\widehat{b}, 2\widehat{c}$ the interior angles of the triangle $p_\alpha p_\beta p_\gamma$ (see Figure \ref{fig:bisector}).
\begin{figure}[ht!]
\labellist
\small\hair 2pt
\pinlabel {${}_{\widehat{a}}$} at 73 49
\pinlabel {${}_{\widehat{a}}$} at 93 49
\pinlabel {${}_{\widehat{a}}$} at 74 42
\pinlabel {${}_{\widehat{a}}$} at 92 43
\pinlabel {${}_{\widehat{b}}$} at 46 65
\pinlabel {${}_{\widehat{b}}$} at 55 71
\pinlabel {${}_{\widehat{b}}$} at 110 70
\pinlabel {${}_{\widehat{b}}$} at 119 65
\pinlabel {${}_{\widehat{c}}$} at 48 29
\pinlabel {${}_{\widehat{c}}$} at 56 22
\pinlabel {${}_{\widehat{c}}$} at 108 22
\pinlabel {${}_{\widehat{c}}$} at 118 29
\pinlabel {$\alpha$} at 85 75
\pinlabel {$\beta$} at 23 62
\pinlabel {$\beta$} at 143 62
\pinlabel {$\gamma$} at 23 33
\pinlabel {$\gamma$} at 143 33
\pinlabel {$\delta$} at 33 49
\pinlabel {$p_\alpha$} at 86 40.5
\pinlabel {$p_\beta$} at 43 77
\pinlabel {$p_\beta$} at 122 77
\pinlabel {$p_\gamma$} at 45 18
\pinlabel {$p_\gamma$} at 120 18
\pinlabel {$\ell_1$} at 63 66
\pinlabel {$\ell_2$} at 103 67
\pinlabel {$\ell_3$} at 101 26
\pinlabel {$\ell_4$} at 63 25
\pinlabel {$t_1$} at 63 80
\pinlabel {$t_2$} at 103 80
\pinlabel {$t_3$} at 101 13
\pinlabel {$t_4$} at 63 13
\endlabellist
\centering
\includegraphics[width=11cm]{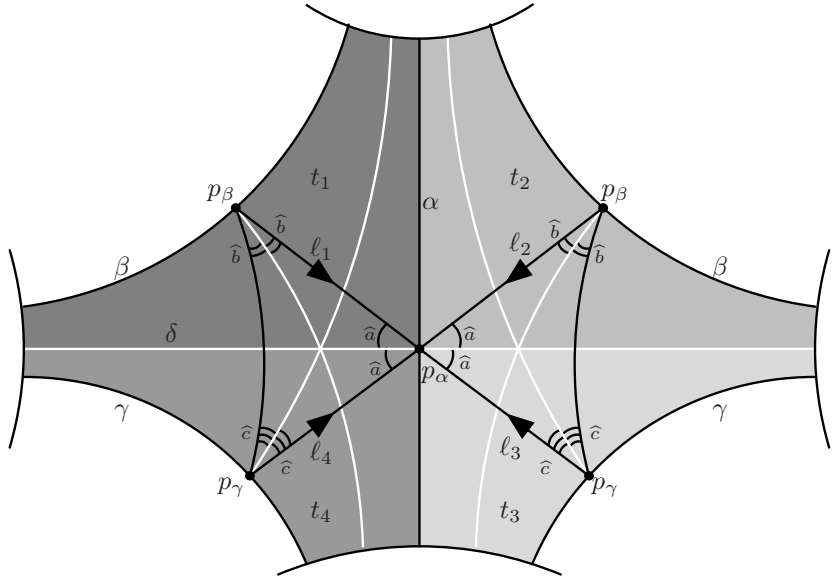}
\caption{Four colored tiles $t_1,\dots, t_4$ of a 3-punctured sphere $S$, in the universal cover. The white lines are heights. The axes $\ell_i$ of all four Killing fields $\psi(t_i)$ run through~$p_\alpha$.}
\label{fig:bisector} 
\end{figure}

The arcs $\beta, \gamma, \alpha, \delta$ subdivide $S$ into four (quotient) tiles $t_1, t_2, t_3, t_4$. Each tile $t_i$ is a right-angled pentagon containing $p_{\alpha}$ as a vertex, and either $p_\beta$ or $p_\gamma$ as an interior point of the opposite edge. Let $\ell_i \subset t_i$ be the segment connecting these two points, oriented towards $p_\alpha$. Assign to each tile $t_i$ the Killing field $\psi(t_i)$ defining a unit-velocity infinitesimal translation along $\ell_i$. Note that $\psi$ respects the symmetry of $S$ defined by reflection in the edges $\alpha, \beta, \gamma$. We claim that $\psi$ (or strictly speaking, its lift to $\mathbb{H}^2$) satisfies the convexity criterion \ref{crit:killing}:
\begin{itemize}
\item Equivariance is true by construction of the lift;
\item Vertex consistency follows from Lemma \ref{lem:bisector}: the points $\psi(t_1)$,\dots, $\psi(t_4)$ form a rectangle in $\mathfrak{psl}_2(\mathbb{R})$, hence in particular a parallelogram;
\item The local increment $\psi(t)-\psi(t')$ across any edge separating tiles $t,t'$ is an (infinitesimal) loxodromy of axis perpendicular to $t\cap t'$, in the correct direction, passing through the correct waist. Indeed:

\smallskip 
\noindent
--- The increment across (either half of) $\alpha$ is, by symmetry, a translation of velocity $A:=2\cos \widehat{a}$, along an axis perpendicular to $\alpha$ at~$p_\alpha$, pushing the adjacent tiles towards each other.
 
\noindent
--- The increment across (either half of) $\delta$ is a translation of velocity $D:=2\sin \widehat{a}$ along an axis perpendicular to $\delta$ at $p_{\delta}=p_\alpha$, pushing the adjacent tiles towards each other.
 
\noindent
--- Using symmetry across $\beta$, the increment at the edge $\beta$ is a translation of velocity $B:=2\cos \widehat{b}$ along an axis perpendicular to $\beta$ at~$p_{\beta}$, pushing the adjacent tiles away from each other.

\noindent
--- Similarly, the increment across $\gamma$ is a translation of velocity $C:=2\cos \widehat{c}$ along an axis perpendicular to $\gamma$ at $p_{\gamma}$, pushing the adjacent tiles away from each other.

\item The convexity inequality \eqref{eq:tocheck} to be checked thus becomes $2\cos \widehat{a} + 2\sin \widehat{a} < 2\cos \widehat{b} + 2\cos \widehat{c}$. This holds true: indeed 
$$\cos \widehat{b} + \cos \widehat{c} > 1 + \cos (\widehat{b}+\widehat{c}) > 1 + \cos (\pi/2 - \widehat{a}) > \cos \widehat{a} + \sin \widehat{a}$$
where the first bound is due to concavity of $\cos$, and the second to $\frac{\pi}{2} > \widehat{a}+\widehat{b}+\widehat{c}$ (since $2\widehat{a}, 2\widehat{b}, 2\widehat{c}$ are the angles of a hyperbolic triangle).
\end{itemize}

This proves Theorem \ref{thm:main} for the thrice punctured sphere.

\section{Proof of Theorem \ref{thm:main} for the once punctured torus} \label{sec:S11}

In the remainder of the paper, $S$ is a once punctured torus. Let $\alpha, \beta, \gamma$ be the edges of a hyperideal triangulation of $S$, and $\delta$ the edge obtained by flipping $\alpha$.

The waist $p_\alpha$ of $\alpha$, still defined as the point $\alpha\cap\delta$, is necessarily fixed under the hyperelliptic involution: $p_\alpha$ is now the midpoint of $\alpha$ \emph{and} of~$\delta$.

\subsection{Loxodromic commutator} \label{sec:loxo}
Let $a,b,c,d$ denote the half-lengths of $\alpha, \beta, \gamma, \delta$. 
Let $S'$ denote the surface $S$ extended by a funnel glued along~$\partial S$.
Place a lift $p$ of $p_\alpha=p_\delta$ at the center of the projective model of $\mathbb{H}^2$ in $\mathbb{P}(\mathbb{R}^{2,1})$. 
Lifts of the edges $\beta, \gamma$ then define a fundamental domain of $S'$, equal to the intersection of $\mathbb{H}^2$ with a parallelogram $\Pi$ (Figure \ref{fig:paral}).
\begin{figure}[ht!]
\labellist
\small\hair 2pt
\pinlabel {$a$} at 32 44
\pinlabel {$a$} at 75 71
\pinlabel {$b$} at 39 18
\pinlabel {$b$} at 81 18
\pinlabel {$b$} at 33 94
\pinlabel {$b$} at 71 94
\pinlabel {$c$} at 9 74
\pinlabel {$c$} at 8 58
\pinlabel {$c$} at 104 55
\pinlabel {$c$} at 103 39
\pinlabel {$d$} at 36 75
\pinlabel {$d$} at 80 44
\pinlabel {$p$} at 55 50.5
\pinlabel {$\Pi$} at 18 18
\pinlabel {$[\mathcal{A}]$} at -1 16
\pinlabel {$[\mathcal{A}']$} at 110 96
\pinlabel {$[\mathcal{D}]$} at 109 16
\pinlabel {$[\mathcal{D}']$} at 5 95
\pinlabel {$p\mathcal{A}\mathcal{D}$} at 55 38
\pinlabel {$p\mathcal{D}\mathcal{A}'$} at 85 56
\pinlabel {$p\mathcal{A}'\mathcal{D}'$} at 56 76
\pinlabel {$p\mathcal{D}'\mathcal{A}$} at 24 57
\pinlabel {$\overline{p\mathcal{A}\mathcal{D}}$} at 59 12
\pinlabel {$\overline{p\mathcal{D}\mathcal{A}'}$} at 117 50
\pinlabel {$\overline{p\mathcal{A}'\mathcal{D}'}$} at 54 98
\pinlabel {$\overline{p\mathcal{D}'\mathcal{A}}$} at -5 58
\pinlabel {${}_{\mathcal{A}\,\mathrm{sh}\, d - \mathcal{D}\,\mathrm{sh}\, a}$} at 209 33
\pinlabel {${}_{\mathcal{D}\,\mathrm{sh}\, d - \mathcal{A}\,\mathrm{sh}\, a}$} at 207 15
\pinlabel {${}_{\mathcal{D}\,\mathrm{sh}\, a - \mathcal{A}'\mathrm{sh}\, d}$} at 237 55
\pinlabel {${}_{\mathcal{A}'\mathrm{sh}\, a - \mathcal{D}\,\mathrm{sh}\, d}$} at 279 52
\pinlabel {${}_{\mathcal{A}'\mathrm{sh}\, d - \mathcal{D}'\mathrm{sh}\, a}$} at 210 81
\pinlabel {${}_{\mathcal{D}'\mathrm{sh}\, d - \mathcal{A}'\mathrm{sh}\, a}$} at 210 97
\pinlabel {${}_{\mathcal{D}'\mathrm{sh}\, a - \mathcal{A}\,\mathrm{sh}\, d}$} at 181 57
\pinlabel {${}_{\mathcal{A}\,\mathrm{sh}\, a - \mathcal{D}'\mathrm{sh}\, d}$} at 139 62
\pinlabel {$[\mathcal{A}]$} at 153 16
\pinlabel {$[\mathcal{A}']$} at 262 96
\pinlabel {$[\mathcal{D}]$} at 261 16
\pinlabel {$[\mathcal{D}']$} at 159 95
\endlabellist
\centering
\includegraphics[width=12.5cm]{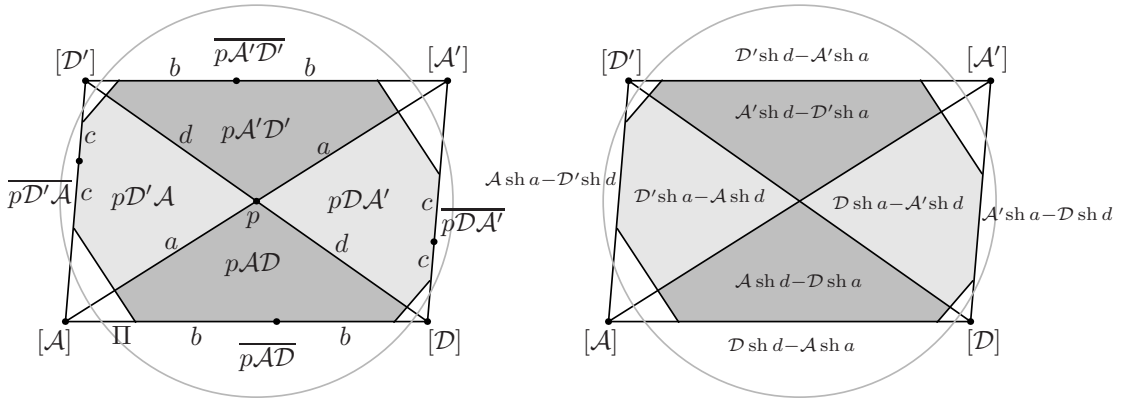}
\caption{Left: lengths in a fundamental domain (right-angled 8-gon) of the punctured torus $S$ made of 4 tiles in~$\mathbb{H}^2$. Dark dots in $\mathbb{H}^2$ are waists. Right: Killing field assignments in the 8 tiles (abbreviating $\sinh$ to $\mathrm{sh}$ and $\cosh$ to $\mathrm{ch}$).}
\label{fig:paral} 
\end{figure}

The boundary of $S$ lifts to lines truncating the corners of $\Pi$. These lines are dual to \emph{unit spacelike} vectors $\mathcal{A},\mathcal{D},\mathcal{A}',\mathcal{D}'$ projecting to the vertices of~$\Pi$, such that $\alpha\subset \mathrm{span} (\mathcal{A},\mathcal{A}')$ and $\delta\subset \mathrm{span} (\mathcal{D},\mathcal{D}')$. 
We may assume that the counterclockwise order of vertics of $\Pi$ goes: $[\mathcal{A}]$, $[\mathcal{D}]$, $[\mathcal{A}']$, $[\mathcal{D}']$.
In $\mathbb{R}^{2,1}$, the third ($p$-parallel) coordinates of $\mathcal{A}$, $\mathcal{D}$, $\mathcal{A}'$, $\mathcal{D}'$ are respectively $\sinh a$, $\sinh d$, $\sinh a$, $\sinh d$; thus
\begin{equation} \label{eq:centered}
(\mathcal{A}+\mathcal{A}')\sinh d = (\mathcal{D}+\mathcal{D}')\sinh a. 
\end{equation}
The lifts of the edges $\alpha, \delta$ subdivide $\Pi\cap \mathbb{H}^2$ into four tiles 
$p\mathcal{A}\mathcal{D}$, 
$p\mathcal{D}\mathcal{A}'$, 
$p\mathcal{A}'\mathcal{D}'$, 
$p\mathcal{D}'\mathcal{A}$ (see Figure \ref{fig:paral}), adjacent respectively to tiles 
$\overline{p\mathcal{A}\mathcal{D}}$, 
$\overline{p\mathcal{D}\mathcal{A}'}$, 
$\overline{p\mathcal{A}'\mathcal{D}'}$, 
$\overline{p\mathcal{D}'\mathcal{A}}$
outside $\Pi$. We pick the following assignment of Killing fields:
\begin{eqnarray*}
\psi(p\mathcal{A}\mathcal{D}) & := & \mathcal{A}\sinh d - \mathcal{D} \sinh a \\
\psi(p\mathcal{D}\mathcal{A}') & := & \mathcal{D}\sinh a - \mathcal{A}' \sinh d \\
\psi(p\mathcal{A}'\mathcal{D}') & := & \mathcal{A}'\sinh d - \mathcal{D}' \sinh a \\
\psi(p\mathcal{D}'\mathcal{A}) & := & \mathcal{D}'\sinh a - \mathcal{A} \sinh d.
\end{eqnarray*}
Note that these are infinitesimal translations whose axes run perpendicular\footnote{Moreover, all four infinitesimal translation axes run through $p$, because all four vectors have vanishing third coordinate; but we will not use this fact.} to the sides of $\Pi$, into $\Pi$, because the vectors on the right-hand side belong to the correct 2-plane quadrants by Fact \ref{fact:classical}.(3). We extend $\psi$ by symmetry under the $\pi$-rotations around the waists (midpoints) of $\beta, \gamma$. (This will in particular force each edge increment, such as $\psi(p\mathcal{A}\mathcal{D})-\psi(\overline{p\mathcal{A}\mathcal{D}})$, to have its axis run through the corresponding edge midpoint, i.e.\ the correct waist.) Note that the $\pi$-rotation around the hyperbolic midpoint of $[\mathcal{A}\mathcal{D}]$, for example, swaps the unit spacelike vectors $\mathcal{A}$ and $\mathcal{D}$, because it swaps the corresponding boundary components of the lift of $S$. This entails
\begin{eqnarray*}
\psi(\overline{p\mathcal{A}\mathcal{D}}) & := & \mathcal{D}\sinh d - \mathcal{A} \sinh a \\
\psi(\overline{p\mathcal{D}\mathcal{A}'}) & := & \mathcal{A}'\sinh a - \mathcal{D} \sinh d \\
\psi(\overline{p\mathcal{A}'\mathcal{D}'}) & := & \mathcal{D}'\sinh d - \mathcal{A}' \sinh a \\
\psi(\overline{p\mathcal{D}'\mathcal{A}}) & := & \mathcal{A}\sinh a - \mathcal{D}' \sinh d.
\end{eqnarray*}

We may now check the convexity criterion \ref{crit:killing} for $\psi$. Equivariance is true by construction.

Consistency at the vertex $\alpha\cap\delta$ is the relationship 
$\psi(p\mathcal{A}\mathcal{D})+
\psi(p\mathcal{A}'\mathcal{D}')=
\psi(p\mathcal{D}\mathcal{A}')+
\psi(p\mathcal{D}'\mathcal{A})$, which follows from \eqref{eq:centered} (actually both sides vanish).

The increment at the edge $\beta$, or $\mathcal{A}\mathcal{D}$, is $\psi(p\mathcal{A}\mathcal{D})-\psi(\overline{p\mathcal{A}\mathcal{D}})= (\mathcal{A}-\mathcal{D})(\sinh a +\sinh d)$, an infinitesimal loxodromy with axis perpendicular to $\mathcal{A}\mathcal{D}$ (at the waist), pulling the tile $p\mathcal{A}\mathcal{D}$ away from $\overline{p\mathcal{A}\mathcal{D}}$, i.e.\ pointing into~$\Pi$. By Fact \ref{fact:classical}, its velocity~is 
$$B:=\Vert \mathcal{A}-\mathcal{D} \Vert (\sinh a + \sinh d) = 2\cosh b \, (\sinh a + \sinh b).$$

The increment at the edge $\gamma$, or $\mathcal{D}\mathcal{A}'$, is $\psi(p\mathcal{D}\mathcal{A}')-\psi(\overline{p\mathcal{D}\mathcal{A}'})= (\mathcal{D}-\mathcal{A}')(\sinh a +\sinh d)$, an infinitesimal loxodromy with axis perpendicular to $\mathcal{D}\mathcal{A}'$, pulling $p\mathcal{D}\mathcal{A}'$ away from $\overline{p\mathcal{D}\mathcal{A}'}$. Its velocity is 
$$C:=\Vert \mathcal{D}-\mathcal{A}' \Vert (\sinh a + \sinh d) = 2\cosh c \, (\sinh a + \sinh b).$$

The increment at the edge $\alpha$, or $p\mathcal{A}$, is $\psi(p\mathcal{A}\mathcal{D})-\psi(p\mathcal{D}'\mathcal{A})=(\mathcal{A}-\mathcal{A}')\sinh d$ (using \eqref{eq:centered}), an infinitesimal loxodromy with axis perpendicular to $\mathcal{A}\mathcal{A}'$, pulling $p\mathcal{A}\mathcal{D}$ towards $p\mathcal{D}\mathcal{A}'$. Its velocity is 
$$A:=\Vert \mathcal{A}-\mathcal{A'} \Vert \sinh d = 2\cosh a \, \sinh d. $$ 

Finally, the increment at the edge $\delta$, or $p\mathcal{D}$, is $\psi(p\mathcal{D}\mathcal{A}')-\psi(p\mathcal{A}\mathcal{D})=(\mathcal{D}-\mathcal{D}')\sinh a$ (using \eqref{eq:centered}), an infinitesimal loxodromy with axis perpendicular to $\mathcal{D}\mathcal{D}'$, pulling $p\mathcal{D}\mathcal{A}'$ towards $p\mathcal{A}\mathcal{D}$. Its velocity is 
$$D:=\Vert \mathcal{D}-\mathcal{D'} \Vert \sinh a = 2\cosh d \, \sinh a.$$ 

It remains to check convexity via \eqref{eq:tocheck}, namely $A+D<B+C$, i.e.\ 
\begin{eqnarray}
\cosh d \, \sinh a + \cosh a \sinh d &<& (\cosh b + \cosh c) (\sinh a + \sinh d) \notag\\
\text{i.e. }\hspace{10pt} \frac{\sinh (a+d)}{\sinh a + \sinh d} &<& \cosh b + \cosh c. \label{haha}
\end{eqnarray}
Let us prove \eqref{haha}. If $\theta$ denotes the angle formed by the diagonals $\alpha$ and $\delta$ of $\Pi$, then a classical trigonometric formula gives (up to permutation)
\begin{eqnarray*}
 \cosh (2b) &=& \sinh a \, \sinh d -\cosh a \, \cosh d\, \cos \theta  \\
 \cosh (2c) &=& \sinh a \, \sinh d +\cosh a \, \cosh d\, \cos \theta.
\end{eqnarray*}
In particular, $\cosh (2b) + \cosh (2c)$ depends only on $a$ and $d$, not on $\theta$.
Since the map $x\mapsto \sqrt{\frac{x+1}{2}}$, taking $\cosh (2u)$ to $\cosh u$, is concave, it follows that the infimal possible value $\mu$ of $\cosh b + \cosh c$ (with $a,d$ fixed) is approached for extremal $\theta$, i.e.\ when $\{\cosh (2b), \cosh (2c) \}=\{1,2\sinh a \, \sinh d -1\}$: thus $\mu=1+\sqrt{\sinh a \, \sinh d}$. The following are equivalent:
\begin{eqnarray*}
  \frac{\sinh (a+d)}{\sinh a + \sinh d} &<& 1+\sqrt{\sinh a \, \sinh d} \\
  \frac{\sinh a\, (2\sinh^2\frac{d}{2}) + \sinh d\, (2\sinh^2\frac{a}{2})}{\sinh a + \sinh d} &<& \sqrt{\sinh a \, \sinh d} \\
  \frac{2\sinh^2\frac{d}{2}}{\sinh d} + \frac{2\sinh^2\frac{a}{2}}{\sinh a} &<& \sqrt{\frac{\sinh a}{\sinh d}} + \sqrt{\frac{\sinh d}{\sinh a}}.  
\end{eqnarray*}
The last inequality is \emph{true}: its left hand side is $\tanh \frac{d}{2}+\tanh \frac{a}{2} <2$, while its right hand side is $\geq 2$. This proves convexity, hence Theorem \ref{thm:main} for $S$ a one-holed torus.

\subsection{Elliptic commutator} \label{sec:comell}
Let $g$ be an incomplete hyperbolic metric on the once-punctured torus $S$ whose completion admits a cone singularity of angle $\theta\in (0,2\pi)$. The holonomy representation of $g$ takes the two generators $u,v$ of $\pi_1(S)$ to two loxodromics with elliptic commutator. In fact, the fixed points of $[u,v]$, $[v,u^{-1}]$, $[u^{-1},v^{-1}]$, $[v^{-1},u]$ in $\mathbb{H}^2$ form the vertices of a convex quadrilateral, equal to a fundamental domain of $(S,g)$ (the generators $u^{\pm 1},v^{\pm 1}$ identify opposite sides in pairs). Any element of the arc complex of $S$ is realized as an embedded geodesic loop $\alpha$ in $S$, connecting the singularity to itself.

We can extend to this context the strip construction along $\alpha$ defined in Section \ref{sec:stripmap}. The main difference is that there are no funnels to extend the metric $g$ into: instead, we should remove from $(S,g)$ a neighborhood of the puncture $p$, then cut along $\alpha$ and insert an appropriate narrow trapezoid of $\mathbb{H}^2$, and finally extend the new metric all the way to a new cone singularity~$p'$. The position of $p'$ is forced by the gluing parameters; see Figure~\ref{fig:newcone}.
\begin{figure}[ht!]
\labellist
\small\hair 2pt
\pinlabel {$\alpha$} at 27 18
\pinlabel {$p$} at 38 66
\pinlabel {$p'$} at 266 74
\endlabellist
\centering
\includegraphics[width=11cm]{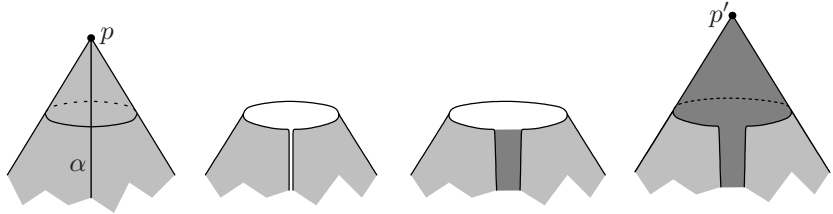}
\caption{Procedure for inserting a strip into a cone metric along an arc $\alpha$. In $S$, since both endpoints of $\alpha$ are at the singularity $p$, we should actually consider a combination of two such procedures.}
\label{fig:newcone} 
\end{figure}

The strip map $\boldsymbol{f}$ is therefore still well-defined, valued in the tangent space at the (smooth) point $[g]$ to the representation variety of $\pi_1(S)$. Thus Conjecture \ref{conj:stripconvex} (convexity of $\boldsymbol{f}$) still makes sense, as does the convexity criterion \ref{crit:killing} (the only difference is that the Killing fields $\psi(\cdot)$ live on the universal cover of the regular part of $S$, which is no longer isometric to $\mathbb{H}^2$: but they still make sense as tilewise Killing fields in the quotient $S$).

\begin{theorem}
Conjecture \ref{conj:stripconvex} continues to hold for $S$ a punctured torus with cone singularity.
\label{thm:ellip}
\end{theorem}
\begin{proof}
We adapt the method from Section \ref{sec:loxo}.
Let $S$ be a hyperbolic punctured torus with cone singularity.
We still call $\alpha, \beta, \gamma$ the edges (running from the singularity to itself) of a triangulation of $S$, and $\delta$ the flip of $\alpha$. The waist of $\alpha$ is its midpoint, where it intersects~$\delta$.

Let $a,b,c,d$ denote the half-lengths of $\alpha, \beta, \gamma, \delta$. 
Place a lift $p$ of $p_\alpha=p_\delta$ at the center of the projective model of $\mathbb{H}^2$ in $\mathbb{P}(\mathbb{R}^{2,1})$. 
Lifts of the edges $\beta, \gamma$ then define a fundamental domain of $S$, equal to a parallelogram $\Pi \subset \mathbb{H}^2$.

Define \emph{unit timelike} vectors $\mathcal{A},\mathcal{D},\mathcal{A}',\mathcal{D}'$ projecting to the vertices of $\Pi$, such that $\alpha\subset \mathrm{span} (\mathcal{A},\mathcal{A}')$ and $\delta\subset \mathrm{span} (\mathcal{D},\mathcal{D}')$. In $\mathbb{R}^{2,1}$, the third ($p$-parallel) coordinates of $\mathcal{A}$, $\mathcal{D}$, $\mathcal{A}'$, $\mathcal{D}'$ are respectively $\cosh a$, $\cosh d$, $\cosh a$, $\cosh d$; thus
\begin{equation} \label{eq:centered-ellip}
(\mathcal{A}+\mathcal{A}')\cosh d = (\mathcal{D}+\mathcal{D}')\cosh a. 
\end{equation}
The lifts of the edges $\alpha, \delta$ subdivide $\Pi$ into four tiles $p\mathcal{A}\mathcal{D}$, $p\mathcal{D}\mathcal{A}'$, $p\mathcal{A}'\mathcal{D}'$, $p\mathcal{D}'\mathcal{A}$, adjacent respectively to $\overline{p\mathcal{A}\mathcal{D}}$, $\overline{p\mathcal{D}\mathcal{A}'}$, $\overline{p\mathcal{A}'\mathcal{D}'}$,  $\overline{p\mathcal{D}'\mathcal{A}}$ (each sharing an edge with $\Pi$). We pick the following assignment of Killing fields (the picture is identical with Figure~\ref{fig:paral}, except $[\mathcal{A}]$, $[\mathcal{D}]$, $[\mathcal{A}']$, $[\mathcal{D}']$ lie inside the disk $\mathbb{H}^2$, and $\cosh$ and $\sinh$ are exchanged):
\begin{eqnarray*}
\psi(p\mathcal{A}\mathcal{D}) & := & \mathcal{A}\cosh d - \mathcal{D} \cosh a \\
\psi(p\mathcal{D}\mathcal{A}') & := & \mathcal{D}\cosh a - \mathcal{A}' \cosh d \\
\psi(p\mathcal{A}'\mathcal{D}') & := & \mathcal{A}'\cosh d - \mathcal{D}' \cosh a \\
\psi(p\mathcal{D}'\mathcal{A}) & := & \mathcal{D}'\cosh a - \mathcal{A} \cosh d.
\end{eqnarray*}
Note that these are infinitesimal translations whose axes run perpendicular to the sides of $\Pi$, because the vectors on the right-hand side belong to the correct 2-plane quadrants (Fact \ref{fact:classical}.(3)). 
We extend $\psi$ by symmetry under the $\pi$-rotations around the waists (midpoints) of $\beta, \gamma$. Note that the $\pi$-rotation around the hyperbolic midpoint of $[\mathcal{A}\mathcal{D}]$, for example, swaps the unit timelike vectors $\mathcal{A}$ and~$\mathcal{D}$. This entails
\begin{eqnarray*}
\psi(\overline{p\mathcal{A}\mathcal{D}}) & := & \mathcal{D}\cosh d - \mathcal{A} \cosh a \\
\psi(\overline{p\mathcal{D}\mathcal{A}'}) & := & \mathcal{A}'\cosh a - \mathcal{D} \cosh d \\
\psi(\overline{p\mathcal{A}'\mathcal{D}'}) & := & \mathcal{D}'\cosh d - \mathcal{A}' \cosh a \\
\psi(\overline{p\mathcal{D}'\mathcal{A}}) & := & \mathcal{A}\cosh a - \mathcal{D}' \cosh d.
\end{eqnarray*}

We may now check the convexity criterion from $\psi$. Equivariance (relative to the holonomy representation of the regular part of $S$) is true by construction. Vertex consistency $\psi(p\mathcal{A}\mathcal{D})+ \psi(p\mathcal{A}'\mathcal{D}')=
\psi(p\mathcal{D}\mathcal{A}')+ \psi(p\mathcal{D}'\mathcal{A})$ follows from \eqref{eq:centered-ellip}.

The increment at the edge $\beta$, or $\mathcal{A}\mathcal{D}$, is $\psi(p\mathcal{A}\mathcal{D})-\psi(\overline{p\mathcal{A}\mathcal{D}})= (\mathcal{A}-\mathcal{D})(\cosh a +\cosh d)$, an infinitesimal loxodromy with axis perpendicular to $\mathcal{A}\mathcal{D}$ (at the waist), pulling the tile $p\mathcal{A}\mathcal{D}$ away from $\overline{p\mathcal{A}\mathcal{D}}$, i.e.\ pointing into~$\Pi$. By Fact \ref{fact:classical}, its velocity~is 
$$B:=\Vert \mathcal{A}-\mathcal{D} \Vert (\cosh a + \cosh d) = 2 \sinh b \, (\cosh a + \cosh d).$$

The increment at the edge $\gamma$, or $\mathcal{D}\mathcal{A}'$, is $\psi(p\mathcal{D}\mathcal{A}')-\psi(\overline{p\mathcal{D}\mathcal{A}'})= (\mathcal{D}-\mathcal{A}')(\cosh a +\cosh d)$, an infinitesimal loxodromy with axis perpendicular to $\mathcal{D}\mathcal{A}'$, pulling $p\mathcal{D}\mathcal{A}'$ away from $\overline{p\mathcal{D}\mathcal{A}'}$. Its velocity is 
$$C:=\Vert \mathcal{D}-\mathcal{A}' \Vert (\cosh a + \cosh d) = 2 \sinh c \, (\cosh a + \cosh d).$$

The increment at the edge $\alpha$, or $p\mathcal{A}$, is $\psi(p\mathcal{A}\mathcal{D})-\psi(p\mathcal{D}'\mathcal{A})=(\mathcal{A}-\mathcal{A}')\cosh d$ (using \eqref{eq:centered-ellip}), an infinitesimal loxodromy with axis perpendicular to $\mathcal{A}\mathcal{A}'$, pulling $p\mathcal{A}\mathcal{D}$ towards $p\mathcal{D}\mathcal{A}'$.
Its velocity is 
$$A:=\Vert \mathcal{A}-\mathcal{A'} \Vert \cosh d = 2\sinh a \, \cosh d.$$ 

Finally, the increment at the edge $\delta$, or $p\mathcal{D}$, is $\psi(p\mathcal{D}\mathcal{A}')-\psi(p\mathcal{A}\mathcal{D})=(\mathcal{D}-\mathcal{D}')\cosh a$ (using \eqref{eq:centered-ellip}), an infinitesimal loxodromy with axis perpendicular to $\mathcal{D}\mathcal{D}'$, pulling $p\mathcal{D}\mathcal{A}'$ towards $p\mathcal{A}\mathcal{D}$.
Its velocity is 
$$D:=\Vert \mathcal{D}-\mathcal{D'} \Vert \cosh a = 2\sinh d \, \cosh a.$$ 
It remains to check convexity via \eqref{eq:tocheck}, namely $A+D<B+C$, i.e.\
\begin{eqnarray}
&& \sinh d \, \cosh a + \sinh a \cosh d \:  < \:  (\sinh b + \sinh c) (\cosh a + \cosh d) \notag \\
&& \text{ i.e. } \hspace{5pt} \frac{\sinh (a+d)}{\cosh a + \cosh d} = \frac{\sinh \frac{a+d}{2}}{\cosh \frac{a-d}{2}} \: < \:  \sinh b + \sinh c. \label{hihi}
\end{eqnarray}
Let us prove \eqref{hihi}. If $\theta$ denotes the angle formed by the diagonals $\alpha$ and $\delta$ of $\Pi$, then a classical trigonometric formula gives (up to permutation)
\begin{eqnarray*}
 \cosh (2b) &=& \cosh a \, \cosh d -\sinh a \, \sinh d\, \cos \theta \\
 \cosh (2c) &=& \cosh a \, \cosh d +\sinh a \, \sinh d\, \cos \theta.
\end{eqnarray*}
In particular, $\cosh (2b) + \cosh (2c)$ depends only on $a$ and $d$, not on $\theta$.
Since the map $x\mapsto \sqrt{\frac{x-1}{2}}$, taking $\cosh (2u)$ to $\sinh u$, is concave, it follows that the infimal possible value of $\sinh b + \sinh c$ (with $a,d$ fixed) is approached when $\theta\rightarrow 0$ or $\theta \rightarrow \pi$, hence $\sinh b + \sinh c \rightarrow \sinh \frac{a+d}{2} + \sinh |\frac{a-d}{2}|$. This is clearly $\geq \left . \sinh \frac{a+d}{2} \right / \cosh \frac{a-d}{2}$ (with equality when $a=d$, but bear in mind that the infimal value is \emph{not} achieved: $\theta\notin\{0,\pi\}$). Theorem \ref{thm:ellip} is proved. \end{proof}

\subsection{Parabolic commutator} \label{sec:parab}
\begin{theorem}
Conjecture \ref{conj:stripconvex} continues to hold for $S$ a one-cusped torus.
\label{thm:parab}
\end{theorem}
\begin{proof}
The case of a cusp (parabolic commutator) can be recovered as a limit case of an elliptic commutator. Namely, given a one-cusped torus $S$ with arcs $\alpha, \beta, \gamma, \delta$ satisfying the combinatorics above, we can find a fundamental domain in $\mathbb{H}^2$ equal to an \emph{ideal} quadrilateral $\Pi$ whose diagonals intersect at~$p$. Denote by $p\mathcal{A}$, $p\mathcal{D}$, $p\mathcal{A}'$, $p\mathcal{D}'$ the diagonal rays issued from $p$, isometrically parameterized (respectively) by functions $m_\mathcal{A}$, $m_\mathcal{D}$, $m_{\mathcal{A}'}$, $m_{\mathcal{D}'}: [0,+\infty)\rightarrow \mathbb{H}^2$. Let $H\subset \mathbb{H}^2$ be the preimage of a fixed small horoball neighborhood of the cusp. Then there exist reals $\overline{a}, \overline{d}>0$ such that $\mathbb{H}^2\smallsetminus H$ contains exactly the initial segment $m_\mathcal{A}([0,\overline{a}])$ (resp.\ $m_\mathcal{D}([0,\overline{d}])$, $m_{\mathcal{A}'}([0,\overline{a}])$, $m_{\mathcal{D}'}([0,\overline{d}])$) of the ray $p\mathcal{A}$ (resp.\ $p\mathcal{D}$, $p\mathcal{A}'$, $p\mathcal{D}'$).

Given $t>0$, the quadrilateral 
$$\Pi_t:=\left ( 
m_\mathcal{A}(\overline{a}+t) \, ,\: 
m_\mathcal{D}(\overline{d}+t) \, ,\: 
m_{\mathcal{A}'}(\overline{a}+t) \, ,\: 
m_{\mathcal{D}'}(\overline{d}+t)  
\right )$$ 
has opposite edges of equal lengths. The isometries taking opposite edges of $\Pi_t$ to one another define a representation $\rho_t:\pi_1(S)\rightarrow \mathrm{PSL}_2(\mathbb{R})$ equal to the holonomy of a cone metric converging to the initial cusped metric as $t\rightarrow +\infty$. Let $a_t, b_t, c_t, d_t$ be the semi-arc lengths in this cone metric; in particular
$a_t=\overline{a}+t$ and $d_t=\overline{d}+t$.

The member ratio of \eqref{hihi} is
$$\frac{\left . \sinh \frac{a_t+d_t}{2}\right / \cosh \frac{a_t-d_t}{2}}{\sinh b_t + \sinh c_t} = 
\frac{\left . \sinh (\frac{\overline{a}+\overline{d}}{2}+t)\right / \cosh \frac{\overline{a}-\overline{d}}{2}}{\sinh b_t + \sinh c_t} <1.$$
To prove convexity of the strip map $\boldsymbol{f}$, we only need to bound this ratio away from $1$ (and take limits as $t\rightarrow +\infty$).
If $\overline{a}\neq\overline{d}$, this comes from the relationship $\sinh b_t + \sinh c_t \geq \sinh \frac{a_t+d_t}{2} + \sinh |\frac{a_t-d_t}{2}|$ proved at the end of  Section \ref{sec:comell}. If $\overline{a}=\overline{d}$, then up to permutation
\begin{eqnarray*}
 \cosh (2b_t) &=& \cosh^2 a_t-\sinh^2 a_t \, \cos \theta = 1+ \sinh^2 a_t \, (1-\cos \theta) \\
 \cosh (2c_t) &=& \cosh^2 a_t+\sinh^2 a_t \, \cos \theta = 1+ \sinh^2 a_t \, (1+\cos \theta) 
\end{eqnarray*}
where $\theta$ is the angle (independent of $t$) formed by the diagonals of $\Pi_t$, hence 
$$\textstyle{\sinh b_t + \sinh c_t = \sinh a_t \Big ( \sqrt{\frac{1-\cos \theta}{2}}+\sqrt{\frac{1+\cos \theta}{2}} \, \Big ) = \sinh a_t \, (\sin \frac{\theta}{2}+\cos \frac{\theta}{2}).}$$
Since $\sin \frac{\theta}{2}+\cos \frac{\theta}{2}>1$, this gives the desired bound.
\end{proof}

\section{Illustration}
\begin{figure}[ht!]
\labellist
\small\hair 2pt
%
\endlabellist
\centering
\includegraphics[width=12cm]{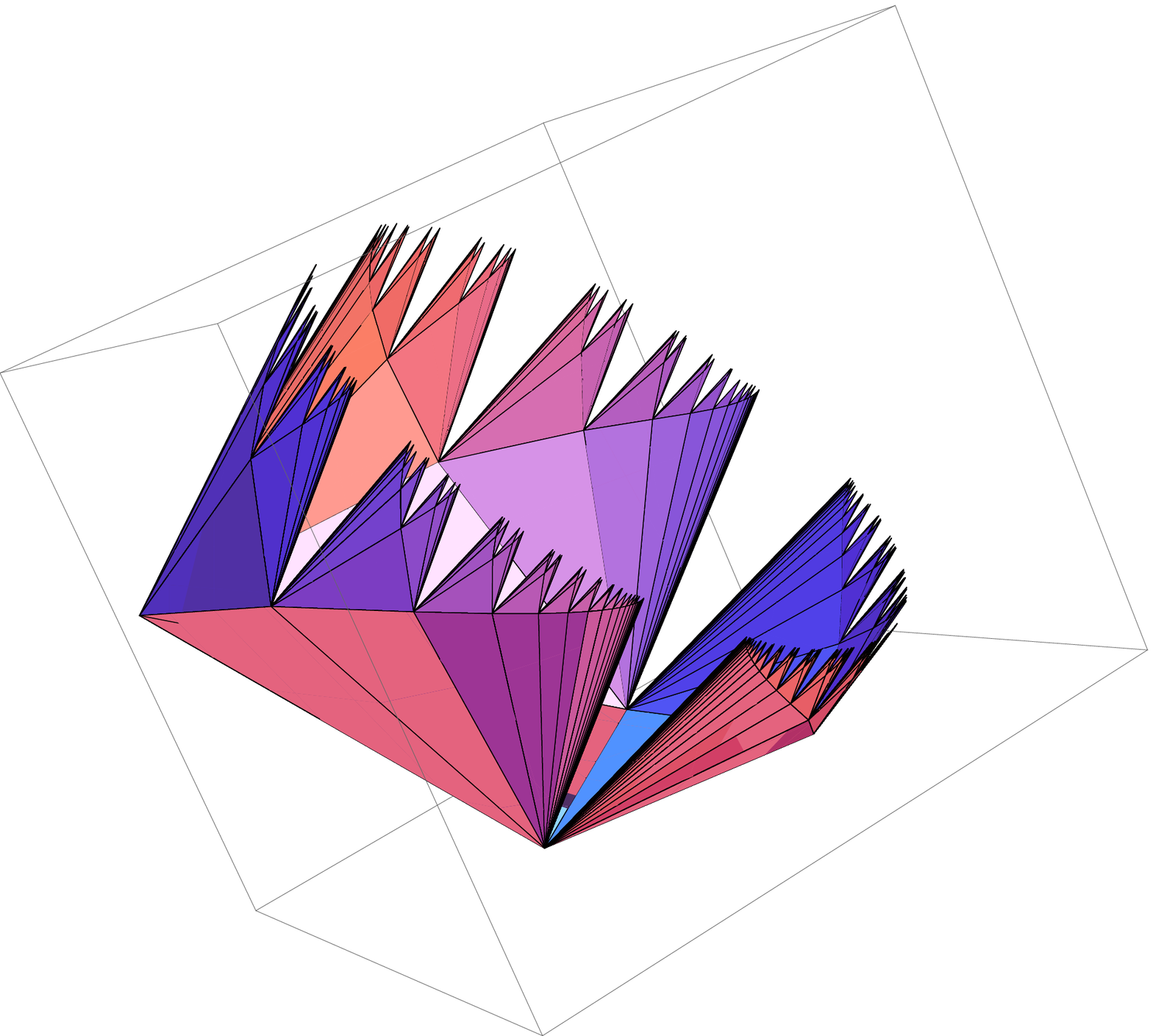}
\caption{}
\label{fig:Nappe} 
\end{figure}

Figure \ref{fig:Nappe} was made using the Mathematica software. It shows the image of the strip map $\boldsymbol{f}:\overline{X}\rightarrow T_{[g]}\mathcal{T}$ for a hyperbolic torus $(S,g)$ with a cone singularity, composed with a projective transformation $\Phi$ of the range $T_{[g]}\mathcal{T}$ sending the origin to infinity. This composition by $\Phi$ enables us to show the whole set $\boldsymbol{f}(\overline{X})$ (which is unbounded in $T_{[g]}\mathcal{T}$). The plane at infinity was sent by $\Phi$ to the plane containing the tips of all the ``teeth''. The gaps between the teeth are not an artefact; they actually grow wider for $g$ a genuine (not conical) hyperbolic metric on~$S$. Each triangular gap lies in a plane containing the point $\Phi(0)$ at infinity.

\vspace{0.5cm}


\begin{thebibliography}{GLMM}
%
%
%
%
%
%
%
%
%
%
%
%
%
\bibitem{dgk2}
\textsc{J. Danciger, F. Gu\'eritaud, F. Kassel}, \textit{Margulis spacetimes via the arc complex}, preprint, http://arxiv.org/abs/1407.5422.
%
%
%
%
%
%
%
%
%
%
%
%
%
\bibitem{PT-note}
\textsc{F. Gu\'eritaud}, \textit{Lengthening deformations of singular hyperbolic tori}, to appear in \emph{Boileau Festschrift} (J.-P.\ Otal, ed.), Ann.\ Fac.\ Sci.\ Toulouse, available at \url{http://math.univ-lille1.fr/~gueritau/math.html}.
%
%
%
%
\bibitem{harer} 
\textsc{J.L.\ Harer}, \textit{The virtual cohomological dimension of the mapping class group of an orientable surface}, Invent. Math. {\bf 84} (1986), 157--176.
%
%
%
\bibitem{ker83} 
\textsc{S.P.\ Kerckhoff}, \textit{The Nielsen realization problem}, Ann. of Math.~117 (1983), p.~235--265.
%

\bibitem{loday} 
\textsc{J.-L.\ Loday}, \textit{Realization of the Stasheff polytope}, Archiv der Mathematik {\bf 83}--3 (2004), 267--278

%
%
%
%
%
%
\bibitem{penner}
\textsc{R. C. Penner}, \textit{The decorated Teichm\"uller space of punctured surfaces}, Comm. Math. Phys.~113 (1987), p.~299--339.
%
%

\bibitem{pt10} 
\textsc{A. Papadopoulos, G. Th\'eret}, \textit{Shortening all the simple closed geodesics on surfaces with boundary}, Proc.\ Amer.\ Math.\ Soc.~138 (2010), p.~1775--1784.


\bibitem{thu86a}
\textsc{W. P. Thurston}, \textit{Minimal stretch maps between hyperbolic surfaces}, preprint (1986), arXiv:9801039.

%
%
\end{thebibliography}
\end{document}